\documentclass[11pt]{article}

\usepackage[top=50mm, bottom=50mm, left=50mm, right=50mm]{geometry}
\usepackage{lineno}
\usepackage{indentfirst}
\usepackage{amssymb}
\usepackage{amsmath}
\usepackage{amsthm}
\usepackage{algorithm}
\usepackage{algorithmicx}
\usepackage{algpseudocode}
\usepackage{appendix}
\usepackage{cases}
\usepackage{epsfig}
\usepackage{graphicx}
\usepackage{graphics}
\usepackage{float}
\usepackage{subfigure}
\usepackage{multirow}
\usepackage{mathrsfs}
\usepackage{color}
\usepackage{fullpage}
\usepackage[normalem]{ulem}
\usepackage{makeidx}
\usepackage{xspace}
\usepackage{wrapfig}
\usepackage{cite}
\usepackage{tabularx}
\usepackage{latexsym}
\usepackage{amsfonts}
\usepackage{epstopdf}
\usepackage{amsthm,amsmath,amssymb}
\usepackage{mathrsfs}
\usepackage{CJK}
\usepackage{indentfirst}
\usepackage{amsmath}
\usepackage{cases}
\usepackage{setspace}
\makeindex

\newtheorem{theorem}{Theorem}[section]

\newtheorem{problem}[theorem]{Problem}
\newtheorem{lemma}[theorem]{Lemma}
\newtheorem{claim}{Claim}

\definecolor{darkred}{rgb}{1, 0.1, 0.3}
\definecolor{darkblue}{rgb}{0.1, 0.1, 1}
\definecolor{darkgreen}{rgb}{0,0.6,0.5}

\newcommand {\mm}[1] {\ifmmode{#1}\else{\mbox{\(#1\)}}\fi}

\floatname{algorithm}{\textbf{Algorithm:}}

\def\qed{\hfill $\Box$\vspace{0.3cm}}
\def\pf{\noindent{\bf Proof. }}
\begin{document}
\title{The saturation number for unions of four cliques}
    \author { {\small Ruo-Xuan Li,~~Rong-Xia Hao
    ,~~Zhen He
    ,~~Wen-Han Zhu}
        \\\\
        \small \it  {School of Mathematics and Statistics, Beijing Jiaotong University, Beijing 100044, P.R. China.}\\
    }
    \date{}
    \maketitle
    \renewcommand{\thefootnote}{}
    \footnotetext {{E-mail addresses}: {\tt 22121589@bjtu.edu.cn} (R.-X. Li), {\tt rxhao@bjtu.edu.cn} (R.-X. Hao), {\tt zhenhe@bjtu.edu.cn} (Z. He), {zhuwenhanzwh@163.com} (W.-H. Zhu).}
\begin{abstract}
A graph $G$ is $H$-saturated if $H$ is not a subgraph of $G$ but $H$ is a subgraph of $G + e$ for any
edge $e$ in $\overline{G}$. The saturation number $sat(n,H)$ for a graph $H$ is the minimal number of edges in
any $H$-saturated graph of order $n$. The $sat(n, K_{p_1} \cup K_{p_2} \cup K_{p_3})$ with $p_3 \ge p_1 + p_2$ was given in [Discrete Math. 347 (2024) 113868]. In this paper, $sat(n,K_{p_1} \cup K_{p_2} \cup K_{p_3} \cup K_{p_4})$ with $p_{i+1} - p_i \ge p_1$ for $2 \le i\le 3$ and $4\le p_1\le p_2$ is determined.
\end{abstract}
    \baselineskip=14pt
\section{Introduction}\label{sec 1}
    Let $G = (V(G), E(G))$ be a simple graph. The vertex set and the edge set of a graph $G$ are denoted by $V(G)$ and $E(G)$. The order and size of $G$ are denoted by $v(G)$ and $e(G)$. Let $E(A,B)$ be the set of the edges between $A$ and $B$ with $A,B \subseteq V(G)$ and $e(A,B)=|E(A,B)|$. For any vertex $v\in V(G)$, the set of neighbors of $v$ in $G$ is denoted by $N_G(v)$ and $N_G [v] = N_G(v) \cup \{v\}$. Furthermore we write $d_G(v) = |N_G(v)|$. A graph is \textit{$k$-regular} if for any vertex $v \in G$, $d_G(v) = k$. Let the minimum degree of $G$ be $\delta(G)$. For $S \subseteq V(G)$, we write $\overline{S} = V (G) \backslash S$ and $G[S]$ is the subgraph of $G$ induced by the vertices in $S$. For any two vertex disjoint graphs $G$ and $H$, $G \cup H$ is the graph with $V (G \cup H) = V (G) \cup V (H)$ and $E(G \cup H) = E(G) \cup E(H)$. The join of the graph $G$ and $H$, denoted by $G \lor H$, is the graph obtained from $G \cup H$ by adding edges between $V (G)$ and $V (H)$. ``w.l.o.g." means ``without loss of generality". For $A,B \subseteq V(G)$, $A\sim B$ means that $A$ and $B$ are completely joint. For more notations and terminologies that will be used in the sequel, we refer to \cite{ref14}, unless otherwise stated.

    Denoted the complement of $G$ by $\overline{G}$. A graph $G$ is $\textit{$H$-saturated}$ if the graph $H$ is not a subgraph of $G$, but for any edge $uv$ in $\overline{G}$, $H$ is a subgraph of $G+uv$. The saturation number for a graph $H$, denoted by $sat(n, H)$, is the minimal number of edges in any $H$-saturated graph of order $n$. The graph $G$ is called an $\textit{extremal graph}$ for $H$, if $G$ is an $H$-saturated graph of order $n$ with $e(G) = sat(n, H)$. Saturation numbers were first studied by P. Erd\H{o}s, A. Hajnal and J. W. Moon in \cite{ref1}. Denoted a complete graph and an independent set of order $n$ by $K_n$ and $I_n$. In \cite{ref1}, it was proved that $sat(n, K_r) = (r-2)(n-r+2) + \binom{r-2}{2}$ and the unique extremal graph for $K_r$ is $K_{r-2} \lor I_{n-r+2}$.

    Let $H(n; p_1, p_2,\dots , p_t) \cong K_{p_1-2} \lor (K_{p_2+1} \cup \dots \cup K_{p_t+1} \cup I_{n-t+3-\sum_{i=1}^t p_i})$ with $2\le  p_1 \le \dots \le  p_t$. L. K¨¢szonyi and Z. Tuza in \cite{ref2} determined $sat(n,tK_2)$ and the extremal graph for $tK_2$. In \cite{ref12}, it has been proved that $sat(n, tK_p) = (p-2)(n-p+2) + \binom{p-2}{2}+(t-1)\binom{p+1}{2}$ and $H(n; p, p, p)$ is the unique $3K_p$-saturated graph. $sat (n, K_{p_1} \cup K_{p_2} \cup K_{p_3})$ ($p_3 \ge p_1 + p_2$) are completely determined in \cite{ref13} and it was proved that $H(n; p_1, p_2, p_3)$ is the extremal graph for $K_{p_1} \cup K_{p_2} \cup K_{p_3}$. In \cite{ref13}, it also proved that $sat(n, K_p\cup (t-1)K_q) = (p-2)(n-p+2) + \binom{p-2}{2}+(t-1)\binom{q+1}{2}$ and $H(n; p, q,\dots , q)$ is a $K_p \cup (t - 1)K_q$-saturated graph. It was shown that $H(n; p, q)$ is the extremal graph for $K_p \cup K_q$ (see Theorem 2.3 in \cite{ref12}). The result that $H(n; p, q, q)$ is the unique extremal graph for $K_p \cup 2K_q$ were given (see Theorem 3.3 in \cite{ref13}). We refer the reader to \cite{ref3,ref4,ref5,ref6,ref7,ref8,ref9,ref10,ref11} for more results about the saturation number.

In this paper, a $K_{p_1} \cup \dots \cup K_{p_t}$-saturated graph is characterized and we have the Theorem~\ref{th2}.  In Section~\ref{sec 3} we will prove the Theorem~\ref{th2}.
    \begin{theorem}\label{th2}
The graph $H(n;p_1, p_2, \dots, p_t)$ is $K_{p_1} \cup \dots \cup K_{p_t}$-saturated if and only if $p_{i+1}-p_i \ge p_1$ or $p_{i+1}=p_i$ for $2\le i\le t-1$.
\end{theorem}
Furthermore, the exact value of $sat(n, K_{p_1} \cup \dots \cup K_{p_4})$ and a extremal graph are determined under some conditions. We will prove the Theorem~\ref{th1} in Section~\ref{sec 4}.
    \begin{theorem}\label{th1}
       Suppose $p_{i+1} - p_i \ge p_1$ for $2 \le i\le 3$, $4\le p_1\le p_2$ and $n > 3(p_1 - 2) +  \sum_{i=2}^{4}p_i(p_i + 1)$. Then $sat(n, K_{p_1} \cup K_{p_2} \cup K_{p_3} \cup K_{p_4})=e(H(n;p_1, p_2, p_3, p_4))=(p_1-2)(n-p_1+2)+\sum_{i=2}^{4}\binom{p_i+1}{2}$ and $H(n;p_1,p_2,p_3,p_4)$ is an extremal graph.
    \end{theorem}
    \section{Preliminary}\label{sec 2}
  First we give several based common properties to all $K_{p_1} \cup \dots \cup K_{p_t}$-saturated graphs.

\begin{lemma}\label{lem1}(\cite{ref13})
Suppose $t \ge 2$, $n > 3(p_1 - 2) +  \sum_{i=2}^{t}p_i(p_i + 1)$. Let $G$ be a $K_{p_1} \cup \dots \cup K_{p_t}$-saturated graph and $e(G) \le e(H(n; p_1, \dots , p_t))$. Let $v$ be a vertex of minimum degree in $V (G)$. Then, we have
\begin{enumerate}
\item[$(1)$] $d_{G}(v) = p_1 - 2$.
\item[$(2)$] $N_{G}(v) \subseteq N_{G}(w)$ for any $w \in \overline{N_{G}(v)}$.
\item[$(3)$] $e(G[\overline{N_{G}(v)}]) \le \sum_{i=2}^{t} \binom{p_i+1}{2}$.
\end{enumerate}
\end{lemma}
Suppose $t \ge 2$, $n > 3(p_1 - 2) + \sum_{i=2}^{t}p_i(p_i + 1)$. Let $G$ be a $K_{p_1} \cup \dots \ K_{p_t}$-saturated graph of order $n$ and $e(G) \le e(H(n; p_1,\dots, p_t))$ and $v$ be a vertex of minimum degree in $G$. Write $S = N_G(v)$. For a vertex $w \in S \backslash \{v\}$, $G + vw$ contains a subgraph $K_{p_1} \cup \dots \cup K_{p_t}$. By Lemma~\ref{lem1}, $d_G(v) = p_{1} - 2$ and then the new edge $vw$ lies in the copy of $K_{p_1}$ induced by $S \cup \{v, w\}$ in $G + vw$. Furthermore, there is a subgraph $K_{p_2} \cup \dots \cup K_{p_t}$ in $G[\overline{S\cup\{v,w\}}]$. In this paper, we always use $H_{vw}=\bigcup\limits_{i=1}^{t-1}H_{vw,i}$ to represent this subgraph $K_{p_2} \cup \dots \cup K_{p_t}$ with $H_{vw,i}\cong K_{p_{i+1}}$ for $i\in [t-1]$.

\begin{lemma}\label{lem2}(\cite{ref13})
Suppose $t \ge 2$, $n > 3(p_1 - 2) +  \sum_{i=2}^{t}p_i(p_i + 1)$. Let $G$ be a $K_{p_1} \cup \dots \cup K_{p_t}$-saturated graph and $e(G) \le e(H(n; p_1, \dots , p_t))$. Let $v$ be a vertex of minimum degree in $V (G)$ and $w\in \overline{S}\backslash \{v\}$. Then,
\begin{enumerate}
\item[$(1)$] there is a subgraph $H_{vw}$ in $G[\overline{S\cup \{v,w\}}]$.
\item[$(2)$] for any vertex $u \in \overline{S} \backslash \{v\}$, if $u$ is adjacent to $w$, then $u \in V(H_{vw})$.
\end{enumerate}
\end{lemma}

 Especially, let $G$ be $K_{p_1} \cup K_{p_2} \cup K_{p_3} \cup K_{p_4}$-saturated. For any vertex $x\in V(H_{vw,i})$, by the fact that $G$ is $K_{p_1} \cup K_{p_2} \cup K_{p_3} \cup K_{p_4}$-saturated, there is a subgraph $H_{vx}=H_{vx,1}\cup H_{vx,2}\cup H_{vx,3}\cong K_{p_2} \cup K_{p_3} \cup K_{p_4}$ in $G[\overline{S\cup\{v,x\}}]$ with $H_{vx,i}\cong K_{p_{i+1}}$ for $i\in [3]$. Let $H$ be an auxiliary graph with $V(H) = V(H_{vw}) \cup V (H_{vx})$ and $E(H) = E(H_{vw}) \cup E(H_{vx})$ (see Figure~\ref{fig1:Auxiliary graph H.}).
            \begin{figure}[H]
	\centering
	\includegraphics[width=0.4\linewidth]{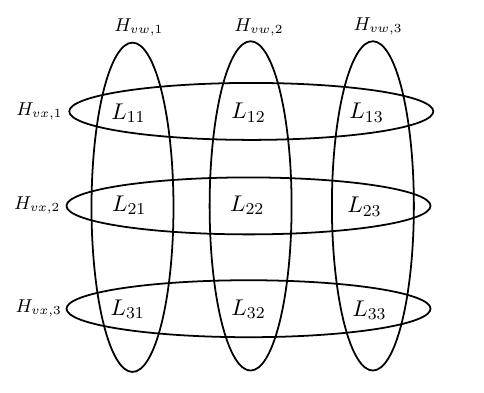} 
	\caption{The auxiliary graph $H$ with $H_{vx,i}\cong K_{p_{i+1}}\cong H_{vw,i}$ for $i\in[3]$.}
	\label{fig1:Auxiliary graph H.}
\end{figure}
            For $1 \le i, j \le 3$, write $V(H_{vx,i}) \cap V(H_{vw,j}) = L_{ij}$ and $|L_{ij}| = {\ell}_{ij}$. Then we have the following Lemma~\ref{lem*}.

\begin{lemma}\label{lem*}
For any vertex $x\in V(H_{vw,i})$ with some $i\in [3]$, we have
\begin{subnumcases} {\label{eq1*}}
p_{j+1}-1 \le {\ell}_{j1}+{\ell}_{j2}+{\ell}_{j3}\le p_{j+1} & for $1\le j\le 3$, {\label{eq11*}}\\
p_{j+1}-1 \le {\ell}_{1j}+{\ell}_{2j}+{\ell}_{3j}\le p_{j+1} & for $j\in\{1,2,3\}\backslash\{i\}$, {\label{eq12*}}\\
{\ell}_{1i}+{\ell}_{2i}+{\ell}_{3i}= p_{i+1}-1. {\label{eq13*}}
\end{subnumcases}
Specifically, if $|V(H_{vx})\backslash V(H_{vw})|=1$, for $k\in\{1,2,3\}\backslash\{i\}$, we have
\begin{align}
{\ell}_{1k}+{\ell}_{2k}+{\ell}_{3k}= p_{k+1}. {\label{eq2*}}
\end{align}
\end{lemma}
\pf
Since $x\in V(H_{vw,i})$ for some $i\in [3]$. The right side of inequalities~(\ref{eq11*}) and (\ref{eq12*}) are clearly true as $H_{vx,i}\cong K_{p_{i+1}}\cong H_{vw,i}$ for $i\in [3]$, so we only need to prove the left side. We first claim that ${\ell}_{j1}+{\ell}_{j2}+{\ell}_{j3}\ge p_{j+1}-1$ for $1\le j\le 3$. To contrary, assume that ${\ell}_{j1}+{\ell}_{j2}+{\ell}_{j3} \le p_{j+1}-2$ for some $j\in[3]$, then there are two adjacent vertices $y$ and $z$ in $G[V(H_{vx,j}) \backslash V(H_{vw})]$. By Lemma~\ref{lem1}(2), $G[S \cup \{y, z\}]\cong K_{p_1}$. Hence, $G[S \cup \{y, z\}] \cup H_{vw}$ is a subgraph isomorphic to $K_{p_1}\cup \dots \cup K_{p_4}$ in $G$, which is a contradiction. Similarly, ${\ell}_{1j}+{\ell}_{2j}+{\ell}_{3j}\ge p_{j+1}-1$ for $1\le j\le 3$. Thus inequalities~(\ref{eq11*}) and (\ref{eq12*}) hold. Since ${\ell}_{1i}+{\ell}_{2i}+{\ell}_{3i}\le |V(H_{vw,i})-\{x\}|$, ${\ell}_{1i}+{\ell}_{2i}+{\ell}_{3i} \le p_{i+1}-1$. Therefore the equality~(\ref{eq13*}) holds.

If $|V(H_{vx})\backslash V(H_{vw})|=1$, then $\sum_{i,j=1}^{3}{\ell}_{ij}=p_2+p_3+p_4-1$. By the equality~(\ref{eq13*}), we have that $\sum\limits_{k\in\{1,2,3\}\backslash\{i\}}({\ell}_{1k}+{\ell}_{2k}+{\ell}_{3k})= p_2+p_3+p_4-p_{i+1}$. Also by the inequality (\ref{eq12*}), the equality~(\ref{eq2*}) holds.

\qed

\begin{lemma}\label{lem**}(\cite{ref13})
$H(n;p,\underbrace{q,\dots, q}_{t-1})$ is a $K_p \cup (t - 1)K_q$-saturated graph for $t\ge 2$ and $2 \le p < q$.
\end{lemma}
    \baselineskip=14pt
    \section{The proof of Theorem~\ref{th2}}\label{sec 3}

\pf
For brevity, write $G = H(n; p_1, \dots, p_t) = H_1 \lor (H_2 \cup \dots
 \cup H_t \cup I_{n-3-t-\sum_{i=1}^{t}p_i})$, where $H_1\cong K_{p_1-2}$ and $H_i\cong K_{p_i+1}$ for $2\le i \le t$.

First, we prove the necessity. If $p_i+1\le p_{i+1}\le p_i+p_1-1$ for some $2\le i\le t-1$, then there is a subgraph $H_{i+1}'\cong K_{p_{i+1}}$ in $G[V(H_1)\cup V(H_i)]$ with $V(H_i)\subseteq V(H_{i+1}')$ and a subgraph $H_i'\cong K_{p_1+p_i}$ in $G[(V(H_1)\cup V(H_{i+1}))\backslash V(H_{i+1}')]$. Also note that there is a subgraph which is isomorphic to $K_{p_2} \cup \dots \cup K_{p_{i-1}}\cup K_{p_{i+2}} \cup \dots \cup K_{p_t}$ in $G[\overline{V(H_1)\cup V(H_i) \cup  V(H_{i+1})}]$. Therefore there is a subgraph which is isomorphic to $K_{p_1} \cup \dots \cup K_{p_t}$ in $G$, a contradiction.

Next, we will prove the sufficiency. For any nonadjacent vertices $u$ and $v$ in $G$, it is clear that $u$ and $v$ lie in distinct $H_i$'s or $u,v\in \overline{V(H_1)\cup \dots \cup V(H_t)}$. Thus, in $G'=G + uv$, we have $G'[V(H_1)\cup \{u,v\}]\cong K_{p_1}$ and there is a subgraph isomorphic to $K_{p_2}\cup \dots \cup K_{p_t}$ in $G[V (G) \backslash (V(H_1) \cup \{u, v\})]$.

We will prove that $G$ is $K_{p_1} \cup \dots \cup K_{p_t}$-free by induction on $t$. By Lemma~\ref{lem**}, we know that $H(n;p_1,p_2)$ is $K_{p_1}\cup K_{p_2}$-free. We assume that $H(n;p_1, \dots , p_{j})$ is $K_{p_1} \cup \dots \cup K_{p_{j}}$-free for $j\le t-1$. For $G=H_1 \lor (H_2 \cup \dots \cup H_t \cup I_{n-3-t-\sum_{i=1}^{t}p_i})$, suppose to the contrary that there is a subgraph $H=G_1\cup \dots \cup G_t \cong K_{p_1} \cup \dots \cup K_{p_t}$ in $G$ and $G_i\cong K_{p_i}$ for $1\le i\le t$. Let $s$ be the smallest number such that $p_s = p_t$. So $p_{s}-p_{s-1}\ge p_1\ge 4$ and therefore $|V(H_1)\lor V(H_k)|\le p_1-2+(p_{s-1}+1)=p_1+p_{s-1}-1<p_s$ for any $2\le k\le s-1$. Also since there is no edge between $H_{t_1}$ and $H_{t_2}$ for $t_1,t_2\in \{2,\dots, t\}$, it is impossible that $G_{a}$ contains vertices from both $H_{t_1}$ and $H_{t_2}$ for $a\in \{s,s+1,\dots, t\}$. Thus $G_{a}$ lies in $H_1\lor H_i$ for $a\in \{s,s+1,\dots,t\}$ and some $s\le i\le t$. Note that $p_i<|V(H_i)|=p_i+1$ for $s\le i\le t$, so there is a subgraph $H'=G_1'\cup \dots \cup G_t'\cong K_{p_1}\cup \dots \cup K_{p_t}$ of $G$ such that $V(G_{a}')\subseteq V(H_{a})$ for each $a\in \{s,s+1,\dots, t\}$. It implies that $H_1\lor (H_2\cup \dots \cup H_{s-1})$ has a subgraph which is isomorphic to $K_{p_1}\cup \dots \cup K_{p_{s-1}}$. According to the inductive hypothesis, we know that $H_1\lor (H_2\cup \dots \cup H_{s-1})$ is $K_{p_1} \cup \dots \cup K_{p_{s-1}}$-free, a contradiction. So $G$ is $K_{p_1}\cup \dots \cup K_{p_t}$-free.\qed
    \baselineskip=14pt
    \section{The proof of Theorem~\ref{th1}}\label{sec 4}
\pf
Let $G$ be an extremal graph for $K_{p_1} \cup K_{p_2} \cup K_{p_3} \cup K_{p_4}$ of order $n > 3(p_1 - 2) +  \sum_{i=2}^{4}p_i(p_i + 1)$ with $p_{i+1} - p_i \ge p_1$ for $2 \le i\le 3$ and $4\le p_1\le p_2$. Suppose that $e(G) < e(H(n; p_1, p_2, p_3, p_4))$. Let $v$ be a vertex of minimum degree in $G$. By Lemma~\ref{lem1} (1), we have $d_{G}(v) = p_1 - 2$. Write $S = N_G(v)$ and $\overline{S} = V (G) \backslash N_G(v)$. By Lemma~\ref{lem1} (2), one has that
\begin{equation}\label{eq1}
e(G[\overline{S}])< \binom{p_2+1}{2}+\binom{p_3+1}{2}+\binom{p_4+1}{2}.
\end{equation}
 Let $w$ be a vertex in $\overline{S}\backslash \{v\}$, by Lemma~\ref{lem2}, we know that $H_{vw} = H_{vw,1} \cup H_{vw,2} \cup H_{vw,3}$ is a subgraph of $G[\overline{S\cup \{v,w\}}]$ with $H_{vw,i} \cong K_{p_{i+1}}$ for $1\le i\le 3$. Firstly, we prove Lemma~\ref{lem3} and Lemma~\ref{lem6} which will be used to prove Theorem~\ref{th1}.
 \begin{lemma}\label{lem3}
    For any vertex $u\in V(H_{vw,k})$ ($1\le k\le 3$), $u$ has at least one neighbor in $V(G) \backslash (S \cup V(H_{vw,k}))$.
    \end{lemma}

\begin{proof}[Proof of Lemma~\ref{lem3}]
        \renewcommand\qedsymbol{$\blacksquare$}
        \renewenvironment{proof}{
            \noindent\textbf{\color{ecolor}\proofname\;} \songti}{
            \hfill $\blacksquare$}
            By the contrary, we assume that there exists a vertex $x_1\in V(H_{vw,k})$ such that $N_{G}[x_1] \subseteq S\cup V(H_{vw,k})$. Since $H_{vw,k}$ is a complete graph, $V(H_{vw,k})\subseteq N_{G}[x_1]$. By Lemma~\ref{lem1}(2), one has that $S\subseteq N_{G}[x_1]$. Thus $N_{G}[x_1]=S\cup V(H_{vw,k})$. For any vertex $x\in V(H_{vw,k})\backslash \{x_1\}$, there is a subgraph $H_{vx}=H_{vx,1}\cup H_{vx,2}\cup H_{vx,3}\cong K_{p_2} \cup K_{p_3} \cup K_{p_4}$ in $G[\overline{S\cup \{v,x\}}]$ with $H_{vx,i}\cong K_{p_{i+1}}$ for $i\in [3]$. By Lemma~\ref{lem2}(2), $x_1\in V(H_{vx})$. For $1 \le i, j \le 3$, write $V(H_{vx,i}) \cap V(H_{vw,j}) = L_{ij}$ and $|L_{ij}| = {\ell}_{ij}$ (see Figure~\ref{fig1:Auxiliary graph H.}). Since $x_1\in V(H_{vw,k})$ for some $1\le k\le 3$, we should consider the following cases.
            \vspace{4mm}

            \noindent \textbf{Case 1.} $k=1$.

            Since $N_{G}[x_1] = S\cup V(H_{vw,1})$ and $H_{vw,1}\cong K_{p_2}$, then $d_{G[\overline{S}\backslash\{x\}]}(x_1)=p_2-2$, so we have $x_1\notin V(H_{vx})$, a contradiction.
            \vspace{4mm}

            \noindent \textbf{Case 2.} $k=3$.

            It is impossible that $x_1\in V(H_{vx,3})$, as $d_{G[\overline{S}\backslash\{x\}]}(x_1)=p_4-2$ and $H_{vx,3}\cong K_{p_4}$. So $x_1\in V(H_{vx,j})$ for $j\in\{1,2\}$. Note that $N_{G[\overline{S}]}[x_1]=V(H_{vw,3})$ and ${\ell}_{j3}\le p_{j+1}$, then $x_1\in L_{j3}$ and ${\ell}_{j3}=p_{j+1}$. Since $H_{vw,3}\cong K_{p_4}$, $|V(H_{vw,3})-L_{j3}|=p_4-p_{j+1}$. So $e(L_{j3},V(H_{vw,3})-L_{j3})\ge \min \{p_3(p_4-p_3),p_2(p_4-p_2)\}$. Hence
            		
\begin{equation}
\begin{aligned}
e(G[\overline{S}])&\ge e(H_{vx,1})+e(H_{vx,2})+e(H_{vx,3})+e(L_{j3},V(H_{vw,3})-L_{j3})
\\&\ge\sum_{i=2}^{4}\binom{p_i}{2}+\min \{p_3(p_4-p_3),p_2(p_4-p_2)\}
\\&>\sum_{i=2}^{4}\binom{p_i}{2}+p_2+p_3+p_4 =\sum_{i=2}^{4}\binom{p_i+1}{2},
\end{aligned}
\nonumber
\end{equation}
a contradiction to Equation~(\ref{eq1}).
\vspace{4mm}

            \noindent \textbf{Case 3.} $k=2$.

             Note that $N_{G[\overline{S}\backslash\{x\}]}[x_1]=V(H_{vw,2})\backslash \{x\}$ and $p_2-1<|V(H_{vw,2})\backslash \{x,x_1\}|=p_3-2< \min\{p_3-1,p_4-1\}$, and since $H_{vx,i}\cong K_{p_{i+1}}$ for $1\le i\le 3$, one has that $x_1\in L_{12}$, ${\ell}_{12}=p_2$ and ${\ell}_{11}={\ell}_{13}=0$.

            Since $x\notin V(H_{vx})$, we have that $\sum_{i=1}^{3}({\ell}_{1i}+{\ell}_{2i}+{\ell}_{3i})\le p_2+p_3+p_4-1$. Note that $\sum_{i=1}^{3}{\ell}_{1i}=p_2$, thus $\sum_{i=1}^{3}({\ell}_{2i}+{\ell}_{3i})\le p_3+p_4-1$. By Inequality~(\ref{eq11*}) of Lemma~\ref{lem*} for $2\le j\le3$, we will distinguish the following three subcases to obtain contradictions.
            \vspace{4mm}

            \noindent \textbf{Subcase 3.1.} ${\ell}_{21}+{\ell}_{22}+{\ell}_{23}=p_3-1$ and ${\ell}_{31}+{\ell}_{32}+{\ell}_{33}=p_4$.

            Since $H_{vx,2}\cong K_{p_3}$ and ${\ell}_{21}+{\ell}_{22}+{\ell}_{23}=p_3-1$, there is only one vertex $u^*\in V(H_{vx})\backslash V(H_{vw})$ and $u^*\in V(H_{vx,2})$. Since $x\in V(H_{vw,2})$, by Equality~(\ref{eq2*}) of Lemma~\ref{lem*} and also ${\ell}_{11}={\ell}_{13}=0$, we have ${\ell}_{21}+{\ell}_{31}=p_2$ and ${\ell}_{23}+{\ell}_{33}=p_4$. As $x\in V(H_{vw,2})$, ${\ell}_{12}=p_2$ and Equality~(\ref{eq13*}) of Lemma~\ref{lem*}, we have that ${\ell}_{22}+{\ell}_{32}=p_3-p_2-1$. By ${\ell}_{21}+{\ell}_{22}+{\ell}_{23}=p_3-1$, we know that $0\le {\ell}_{23} \le p_3-1$.

             We will prove that $1\le {\ell}_{23}\le p_3-1$. Suppose that ${\ell}_{23}=0$, since ${\ell}_{21}+{\ell}_{22}+{\ell}_{23}=p_3-1$, ${\ell}_{21}\le p_2$ and ${\ell}_{22}\le p_3-p_2-1$, one has that ${\ell}_{21}=p_2$ and ${\ell}_{22}=p_3-p_2-1$. Thus there are at least ${\ell}_{21}{\ell}_{22}=p_2(p_3-p_2-1)$ edges between $V(H_{vw,1})$ and $V(H_{vw,2})$. Also, $e(u^*,V(H_{vw,1}) \cup V(H_{vw,2}))\ge p_2+p_3-p_2-1=p_3-1$ as $V(H_{vw,1})\cup (V(H_{vw,2})\backslash (\{x\}\cup L_{12}))\subseteq N_{G[\overline{S}]}(u^*)$. By the proof of Case 2 of Lemma~\ref{lem3}, we know that for any vertex $u\in V(H_{vw,3})$, $u$ has at least one neighbor in $V(G) \backslash (S \cup V(H_{vw,3}))$. It implies that $e(V(H_{vw,3}),\overline{S\cup V(H_{vw,3})})\ge p_4$. Hence
            \begin{equation}
            \begin{aligned}
            e(G[\overline{S}])&\ge e(H_{vw,1})+e(H_{vw,2})+e(H_{vw,3})+e(V(H_{vw,1}),V(H_{vw,2}))+e(u^*,V(H_{vw,1}) \cup V(H_{vw,2}))
            \\&+e(V(H_{vw,3}),\overline{S\cup V(H_{vw,3})})
            \\&\ge\sum_{i=2}^{4}\binom{p_i}{2}+p_2(p_3-p_2-1)+(p_3-1)+p_4
            \\&>\sum_{i=2}^{4}\binom{p_i}{2}+p_2+p_3+p_4,
            \end{aligned}
            \nonumber
            \end{equation}
            a contradiction to Equation (\ref{eq1}).

            Following that, we consider $1\le {\ell}_{23}\le p_3-1$. Recall that ${\ell}_{21}+{\ell}_{22}+{\ell}_{23}=p_3-1$, ${\ell}_{31}+{\ell}_{32}+{\ell}_{33}=p_4$ and ${\ell}_{23}+{\ell}_{33}=p_4$. We have
            \begin{equation}
            \begin{aligned}
            e(G[\overline{S}])&\ge e(H_{vw,1})+e(H_{vw,2})+e(H_{vw,3})+{\ell}_{23}({\ell}_{21}+{\ell}_{22})+{\ell}_{33}({\ell}_{31}+{\ell}_{32})+e(u^*,V(H_{vx,2})\backslash \{u^*\})
            \\&=\sum_{i=2}^{4}\binom{p_i}{2}+{\ell}_{23}(p_3-1-{\ell}_{23})+{\ell}_{33}(p_4-{\ell}_{33})+p_3-1
            \\&=\sum_{i=2}^{4}\binom{p_i}{2}+{\ell}_{23}(p_3-1-{\ell}_{23})+(p_4-{\ell}_{23}){\ell}_{23}+p_3-1.
            \end{aligned}
            \nonumber
            \end{equation}
            Let $f(x)=x(p_3-1-x)+(p_4-x)x+p_3-1$ ($1\le x \le p_3-1$), then $f(x)\ge \min\{f(1),f(p_3-1)\}=f(1)=2p_3+p_4-4\ge p_2+p_3+p_4$. Therefore, $e(G[\overline{S}])\ge \sum_{i=2}^{4}\binom{p_i}{2}+p_2+p_3+p_4=\sum_{i=2}^{4}\binom{p_i+1}{2}$, a contradiction to Equation (\ref{eq1}).
            \vspace{4mm}

            \noindent \textbf{Subcase 3.2.} ${\ell}_{21}+{\ell}_{22}+{\ell}_{23}=p_3$ and ${\ell}_{31}+{\ell}_{32}+{\ell}_{33}=p_4-1$.

             Since $H_{vx,3}\cong K_{p_4}$ and ${\ell}_{31}+{\ell}_{32}+{\ell}_{33}=p_4-1$, there is only one vertex $u'\in V(H_{vx})\backslash V(H_{vw})$ and $u'\in V(H_{vx,3})$. Note that ${\ell}_{11}=0$, ${\ell}_{12}=p_2$ and $x\in V(H_{vw,2})$, we have that ${\ell}_{21}+{\ell}_{31}+{\ell}_{22}+{\ell}_{32}\le p_3-1$. It implies that $0\le {\ell}_{31}+{\ell}_{32}\le p_3-1$. If ${\ell}_{31}={\ell}_{32}=0$, by Equalities~(\ref{eq13*}) and (\ref{eq2*}) of Lemma~\ref{lem*}, we have ${\ell}_{22}=p_3-p_2-1$ and ${\ell}_{21}=p_2$. Therefore by the condition of this subcase, we have ${\ell}_{23}=1$ and ${\ell}_{33} = p_4-1$. Hence
            \begin{equation}
            \begin{aligned}
            e(G[\overline{S}])&\ge e(H_{vw,1})+e(H_{vw,2})+e(H_{vw,3})+{\ell}_{21}{\ell}_{22}+{\ell}_{23}({\ell}_{21}+{\ell}_{22})+e(u',V(H_{vx,3})\backslash \{u'\})
            \\&=\sum_{i=2}^{4}\binom{p_i}{2}+p_2(p_3-p_2-1)+(p_3-1)+(p_4-1)
            \\&>\sum_{i=2}^{4}\binom{p_i}{2}+p_2+p_3+p_4=\sum_{i=2}^{4}\binom{p_i+1}{2},
            \end{aligned}
            \nonumber
            \end{equation}
            a contradiction to Equation (\ref{eq1}).

            Thus, $1 \le {\ell}_{31}+{\ell}_{32} \le p_3-1$. For simplicity, let ${\ell}_{31}+{\ell}_{32}=a$, then ${\ell}_{33}=p_4-1-a$. Note that $x\in V(H_{vw,2})$ and ${\ell}_{13}=0$, by Equality~(\ref{eq2*}) of Lemma~\ref{lem*}, we know that ${\ell}_{23}=a+1$. Therefore
            \begin{equation}
            \begin{aligned}
            e(G[\overline{S}])&\ge e(H_{vw,1})+e(H_{vw,2})+e(H_{vw,3})+{\ell}_{23}({\ell}_{21}+{\ell}_{22})+{\ell}_{33}({\ell}_{31}+{\ell}_{32})+e(u',V(H_{vx,3})\backslash \{u'\})
            \\&=\sum_{i=2}^{4}\binom{p_i}{2}+(a+1)(p_3-1-a)+(p_4-1-a)a+p_4-1.
            \end{aligned}
            \nonumber
            \end{equation}
            Let $g(x)=(x+1)(p_3-1-x)+(p_4-1-x)x+p_4-1$ ($1\le x \le p_3-1$), then $g(x)\ge \min\{g(1),g(p_3-1)\}=g(1)=2p_3+2p_4-7> p_2+p_3+p_4$. Therefore, $e(G[\overline{S}])> \sum_{i=2}^{4}\binom{p_i}{2}+p_2+p_3+p_4=\sum_{i=2}^{4}\binom{p_i+1}{2}$, a contradiction to Equation (\ref{eq1}).
            \vspace{4mm}

            \noindent \textbf{Subcase 3.3.} ${\ell}_{21}+{\ell}_{22}+{\ell}_{23}=p_3-1$ and ${\ell}_{31}+{\ell}_{32}+{\ell}_{33}=p_4-1$.

            As $H_{vx,2}\cong K_{p_3}$, $H_{vx,3}\cong K_{p_4}$ and the condition of this subcase, there are two vertices $u' \in V(H_{vx,2})\backslash V(H_{vw})$ and $u'' \in V(H_{vx,3})\backslash V(H_{vw})$, respectively.

            If ${\ell}_{23}=0$, by Inequality~(\ref{eq12*}) of Lemma~\ref{lem*} and ${\ell}_{13}=0$, we have ${\ell}_{33}\ge p_4-1$. Furthermore, recall that ${\ell}_{31}+{\ell}_{32}+{\ell}_{33}=p_4-1$, then ${\ell}_{33} = p_4-1$ and therefore ${\ell}_{31}={\ell}_{32}=0$. Since $x\in V(H_{vw,2})$, ${\ell}_{11}=0$ and ${\ell}_{12}=p_2$, by inequalities~(\ref{eq12*}) and (\ref{eq13*}) of Lemma~\ref{lem*}, ${\ell}_{21}\ge p_2-1$ and ${\ell}_{22}=p_3-p_2-1$. Hence
            \begin{equation}
            \begin{aligned}
            e(G[\overline{S}])&\ge e(H_{vw,1})+e(H_{vw,2})+e(H_{vw,3})+{\ell}_{21}{\ell}_{22}+e(u',V(H_{vx,2})\backslash \{u'\})+e(u'',V(H_{vx,3})\backslash \{u''\})
            \\&=\sum_{i=2}^{4}\binom{p_i}{2}+(p_2-1)(p_3-p_2-1)+(p_3-1)+(p_4-1)
            \\&>\sum_{i=2}^{4}\binom{p_i}{2}+p_2+p_3+p_4=\sum_{i=2}^{4}\binom{p_i+1}{2},
            \end{aligned}
            \nonumber
            \end{equation}
            a contradiction to Equation (\ref{eq1}).

            If ${\ell}_{23} \ge 1$. We have ${\ell}_{23}\le p_3-1$ due to ${\ell}_{21}+{\ell}_{22}+{\ell}_{23}=p_3-1$. Thus by the left side of Inequality~(\ref{eq12*}) of Lemma~\ref{lem*} and ${\ell}_{13}=0$, we have ${\ell}_{33}\ge p_4-p_3\ge 1$. Hence
            \begin{equation}
            \begin{aligned}
            e(G[\overline{S}])&\ge e(H_{vw,1})+e(H_{vw,2})+e(H_{vw,3})+{\ell}_{23}({\ell}_{21}+{\ell}_{22})+{\ell}_{33}({\ell}_{31}+{\ell}_{32})
            \\&+e(u',V(H_{vx,2})\backslash \{u'\})+e(u'',V(H_{vx,3})\backslash \{u''\})
            \\&\ge\sum_{i=2}^{4}\binom{p_i}{2}+({\ell}_{21}+{\ell}_{31})+({\ell}_{31}+{\ell}_{32})+p_3-1+p_4-1
            \\&\ge\sum_{i=2}^{4}\binom{p_i}{2}+(p_2-1)+(p_3-p_2-1)+(p_3-1)+(p_4-1)
            \\&\ge\sum_{i=2}^{4}\binom{p_i}{2}+p_2+p_3+p_4=\sum_{i=2}^{4}\binom{p_i+1}{2},
            \end{aligned}
            \nonumber
            \end{equation}
            a contradiction to Equation (\ref{eq1}).

\end{proof}

\begin{lemma}\label{lem6}
    Let $A=\{u|d_{G[\overline{S}]}(u)\ge 1,u\in \overline{S}\}$. Then $|A|\ge p_2+p_3+p_4+2$ and $d_{G[\overline{S}]}(u)\ge p_2-1$ for any $u\in A$.
    \end{lemma}

\begin{proof}[Proof of Lemma~\ref{lem6}]
        \renewcommand\qedsymbol{$\blacksquare$}
        \renewenvironment{proof}{
            \noindent\textbf{\color{ecolor}\proofname\;} \songti}{
            \hfill $\blacksquare$}
 If $G[V(H_{vw})]=K_{p_2+p_3+p_4}$, then $e(G[\overline{S}])= \binom{p_2+p_3+p_4}{2}> \sum_{i=2}^{4}\binom{p_i+1}{2}$, a contradiction to Equation (\ref{eq1}). So there exist two vertices $v_1$ and $v_2$ such that $v_1$ and $v_2$ belong to two different cliques of $H_{vw}$ and $v_1v_2\notin E(G)$. Consider the graph $G+v_1v_2$ and there is a subgraph $H=H_1\cup H_2 \cup H_3 \cup H_4$ of $G+v_1v_2$ with $H_i\cong K_{p_i}$ for $i\in [4]$. Note that $|S\cup V(H_{vw})|=p_1+p_2+p_3+p_4-2$, then there are at least two vertices $x_1$ and $x_2$ of $\overline{S}\backslash V(H_{vw})$ in $H$. Therefore $d_G(x_i)\ge p_1-1$ and $d_{G[\overline{S}]}(x_i)\ge 1$ as $|S|=p_1-2$ for $i=1,2$. Hence $x_1,x_2\in A$. Furthermore, since $d_{G[{\overline{S}}]}(x)\ge p_2-1>1$ for any $x\in V(H_{vw})$, $V(H_{vw})\subseteq A$.  So $|A|\ge p_2+p_3+p_4+2$.

For any vertex $y\in A$, choose a vertex $y_1\in N_{G[\overline{S}]}(y)$ and there is a subgraph $G[S\cup \{v,y_1\}]\cup H'$ in $G+vy_1$ with $G[S\cup \{v,y_1\}]\cong K_{p_1}$ and $H'\cong K_{p_2}\cup K_{p_3}\cup K_{p_4}$. We have that $y$ must be in $H'$, otherwise $G[S\cup \{y,y_1\}]\cup H'$ is a subgraph of $G$ which is isomorphic to $K_{p_1}\cup K_{p_2}\cup K_{p_3}\cup K_{p_4}$, a contradiction. It implies that $d_{G[\overline{S}]}(y)\ge p_2-1$.

\end{proof}

The following two claims are crucial for the proof of Theorem~\ref{th1}.

\begin{claim}\label{lem4}
   For $4\le p_1\le p_2$, let $G^*$ be a graph such that $H=H_1\lor (H_2\cup H_3)$ is a subgraph of $G^*$ with $H_1\cong K_{p_1-2}$, $H_2\cong K_{p_2}$ and $H_3$ is an empty graph with $V(H_3)=\overline{V(H_1)\cup V(H_2)}$, but $K_{p_1}\cup K_{p_2}$ is not a subgraph of $G^*$; also there is a vertex $v$ such that $v\in V(H_3)$ and $N_{G^*}(v)=V(H_1)$. If there is a subgraph which is isomorphic to $K_{p_1}\cup K_{p_2}$ in $G^*+vw'$ with any $w'\in \overline {V(H_1)\cup \{v\}}$, then $e(\overline{H_1}) \ge \binom{p_2+1}{2}$.
    \end{claim}

\begin{proof}[Proof of Claim~\ref{lem4}]
        \renewcommand\qedsymbol{$\blacksquare$}
        \renewenvironment{proof}{
            \noindent\textbf{\color{ecolor}\proofname\;} \songti}{
            \hfill $\blacksquare$}
            For convenience, let $S_1=V(H_1)$. For a vertex $u_1\in V(H_2)$, by the condition of this claim, we know that there is a subgraph $G_1\cup G_2\cong K_{p_1}\cup K_{p_2}$ in $G'=G^*+u_1v$ such that $G_i\cong K_{p_i}$ with $i\in [2]$. Since $N_{G^*}(v)=S_1$, $G'[S_1\cup \{u_1,v\}]=G_1$. Note that $K_{p_1}\cup K_{p_2}$ is not a subgraph of $G^*$ and $S_1\subseteq N_{G^*}(w')$ for any $w'\in \overline {S_1}$, so $w_1w_2\notin E(G^*)$ for any $w_1,w_2 \in V(H_3)$, which implies that there is at most one vertex of $V(H_3)$ in $V(G_2)$ and $V(H_2)\backslash\{u_1\}\subseteq V(G_2)$. So there is a vertex $u$ of $V(H_3)$ in $V(G_2)$ and $\{u\}\sim V(H_2)\backslash \{u_1\}$. Thus $e(\overline{H_1}) \ge \binom{p_2}{2}+p_2-1$. If $e(\overline{H_1}) = \binom{p_2}{2}+p_2-1$, then $E(\overline{H_1}) = E(H_2)\cup E(u, V(H_2)\backslash \{u_1\})$. For $u_2\in V(H_2)\backslash \{u_1\}$, $K_{p_1}\cup K_{p_2}$ is not a subgraph of $G^*+vu_2$, a contradiction. So $e(\overline{H_1}) \ge \binom{p_2+1}{2}$.
\end{proof}

\begin{claim}\label{lem5}
   For $4\le p_1\le p_2$ and $p_3-p_2\ge p_1$, let $G^*$ be a graph such that $H=H_1\lor (H_2\cup H_3\cup H_4)$ is a subgraph of $G^*$ with $H_1\cong K_{p_1-2}$, $H_i\cong K_{p_i}$ for $i\in \{2,3\}$ and $H_4$ is an empty graph with $V(H_4)=\overline{V(H_1)\cup V(H_2)\cup V(H_3)}$, but $K_{p_1}\cup K_{p_2}\cup K_{p_3}$ is not a subgraph of $G^*$; also there is a vertex $v$ such that $v\in V(H_4)$ and $N_{G^*}(v)=V(H_1)$. If there is a subgraph which is isomorphic to $K_{p_1}\cup K_{p_2}\cup K_{p_3}$ in $G^*+vw'$ for any $w'\in \overline {V(H_1)\cup \{v\}}$, then $e(\overline{H_1}) \ge \binom{p_2+1}{2}+\binom{p_3+1}{2}$.
    \end{claim}

\begin{proof}[Proof of Claim~\ref{lem5}]
        \renewcommand\qedsymbol{$\blacksquare$}
        \renewenvironment{proof}{
            \noindent\textbf{\color{ecolor}\proofname\;} \songti}{
            \hfill $\blacksquare$}
For convenience, let $S_1=V(H_1)$. If there is a vertex $u_1\in V(H_3)$ satisfying that $N_{G^*}[u_1]=S_1\cup V(H_3)$, we will consider the graph $G^*+vx$ with $x\in V(H_3)\backslash \{u_1\}$. By the condition of this claim, there is a subgraph $G_1\cup G_2\cup G_3$ with $G_i\cong K_{p_i}$ in $G'=G^*+vx$. Since $N_{G^*}(v) = S_1$, $G'[S_1\cup \{v,x\}]=G_1$. It implies that $V(G_2\cup G_3)\subseteq V(\overline{S_1\cup \{v,x\}})$. It can be proved that $u_1\in V(G_2\cup G_3)$, otherwise $G^*[S_1\cup \{u_1,x\}]\cup G_2\cup G_3\cong K_{p_1}\cup K_{p_2}\cup K_{p_3}$ in $G^*$, a contradiction. Note that $p_2<d_{G^*\backslash V(G_1)}[u_1]=p_3-1<p_3$ as $N_{G^*}[u_1]=S_1\cup V(H_3)$, so $u_1\in V(H_3)\cap V(G_2)$ and $|V(H_3)\cap V(G_2)|=p_2$. Since $H_3\cong K_{p_3}$, we have that $|V(H_3)|-|V(H_3)\cap V(G_2)|=p_3-p_2$. Hence
\begin{equation}
\begin{aligned}
e(\overline{H_1})&\ge e(G_2)+e(G_3)+|V(H_3)\cap V(G_2)|\times(|V(H_3)|-|V(H_3)\cap V(G_2)|)
\\&=\binom{p_2}{2}+\binom{p_3}{2}+p_2(p_3-p_2)
\\&>\binom{p_2}{2}+\binom{p_3}{2}+p_2+p_3 =\binom{p_2+1}{2}+\binom{p_3+1}{2}.
\end{aligned}
\nonumber
\end{equation}
Then the claim holds. Next we only consider that $N_{G^*}(u)\cap V(\overline{H_1\cup H_3})\neq \emptyset$ for any $u\in V(H_3)$.

If for any $w'\in \overline{S_1\cup V(H_3)\cup \{v\}}$, there exists a subgraph $G_{vw',1}\cup G_{vw',2}\cup G_{vw',3}$ in $G^*+vw'$ with $G_{vw',i}\cong K_{p_i}$ for $i\in[2]$ and $G_{vw',3}=H_3$, we will consider the graph $G'=G^*[V(G^*)-V(H_3)]$. Note that $K_{p_1-2}\cup K_{p_2}$ is a subgraph of $G'$, but $K_{p_1}\cup K_{p_2}$ is not, and $G_{vw',1}\cup G_{vw',2}\cong K_{p_1}\cup K_{p_2}$ in $G'+vw'$ with $w'\in V(G')-(V(H_1)\cup \{v\})$. By Claim~\ref{lem4}, we know that $e(G'[V(G')-V(H_1)])\ge \binom{p_2+1}{2}$ and therefore $e(G^*[\overline{V(H_1)\cup V(H_3)}])\ge \binom{p_2+1}{2}$. Also since $N_{G^*}(u)\cap (V(\overline{H_1\cup H_3}))\neq \emptyset$ for any $u\in V(H_3)$, one has that there are at least $|V(H_3)|=p_3$ edges between $V(H_3)$ and $V(\overline{H_1\cup H_3})$ in $G^*$. Hence
\begin{equation}
\begin{aligned}
e(\overline{H_1})&\ge e(G^*[\overline{V(H_1)\cup V(H_3)}])+e(H_3)+e(V(H_3),V(\overline{H_1\cup H_3}))
\\&\ge\binom{p_2+1}{2}+\binom{p_3}{2}+p_3
\\&=\binom{p_2+1}{2}+\binom{p_3+1}{2}.
\end{aligned}
\nonumber
\end{equation}
So the claim holds. Therefore we can assume that there is a $w'\in \overline{S_1\cup V(H_3)\cup \{v\}}$ such that $|V(G_{vw',3})\cap V(H_3)|\le p_3-1$ for any choice of $G_{vw',1}\cup G_{vw',2}\cup G_{vw',3}$ in $G^*+vw'$ with $G_{vw',i}\cong K_{p_i}$ and $i\in[3]$.

If $|V(G_{vw',3})\cap V(H_3)|= p_3-1$ for any choice of $G_{vw',1}\cup G_{vw',2}\cup G_{vw',3}$ in $G^*+vw'$, let $V(G_{vw',3})\backslash V(H_3)=\{u_1\}$ and $V(H_3)\backslash V(G_{vw',3})=\{u_2\}$. Then $u_2\in V(G_{vw',2})$, otherwise $G'_{vw'}=G^*[S_1\cup \{v,w'\}]\cup G_{vw',2} \cup H_3$ is also a copy of $K_{p_1}\cup K_{p_2}\cup K_{p_3}$ in $G^*+vw'$, a contradiction to the choice of $w'$. So $u_2\in V(H_3)\cap V(G_{vw',2})$ and since $|V(G_{vw',3})\cap V(H_3)|= p_3-1$, we have that $V(H_3)\cap V(G_{vw',2})=\{u_2\}$.

Subsequently, we first consider $V(G_{vw',2})\backslash \{u_2\} \nsubseteq V(H_2)$.
\begin{equation}
\begin{aligned}
e(\overline{H_1})&\ge e(H_2)+e(H_3)+ e(u_1,V(G_{vw',3})\backslash \{u_1\})+e(u_2,V(G_{vw',2})\backslash \{u_2\})
\\&+|E(G^*[V(G_{vw',2})\backslash \{u_2\}])\backslash E(H_2)|
\\&\ge\binom{p_2}{2}+\binom{p_3}{2}+(p_3-1)+(p_2-1)+(p_2-2)
\\&\ge\binom{p_2+1}{2}+\binom{p_3+1}{2}.
\end{aligned}
\nonumber
\end{equation}
So the claim holds. Next we consider $V(G_{vw',2})\backslash \{u_2\} \subseteq V(H_2)$, that is $|V(G_{vw',2})\cap V(H_2)|=p_2-1$. If there is a vertex $x\in \overline{V(H_3)\cup \{u_1\}}$ with degree at least $p_3-1$ in $G^*[\overline{S_1}]$, then
\begin{equation}
\begin{aligned}
e(\overline{H_1})&\ge e(H_2)+e(H_3)+ e(u_1,V(G_{vw',3})\backslash \{u_1\})+e(u_2,V(G_{vw',2})\backslash \{u_2\})+e(x,\overline{S_1}\backslash V(H_2))
\\&\ge \binom{p_2}{2}+\binom{p_3}{2}+(p_3-1)+(p_2-1)+(p_3-1)-p_2
\\&>\binom{p_2+1}{2}+\binom{p_3+1}{2}.
\end{aligned}
\nonumber
\end{equation}
Then the claim holds. Then we consider that $d_{G^*[\overline{S_1}]}(x)<p_3-1$ for any $x\in \overline{S_1\cup V(H_3)\cup \{u_1\}}$. It implies that $G^*[\{u_1\}\cup (V(H_3)\backslash \{v_1\})]$ is the only possible subgraph which is isomorphic to $K_{p_3}$ in $G^*+vv_1$ for $v_1\in V(H_3)\backslash \{u_2\}$. Recall that $u_2\in V(H_3)$, then $u_1u_2\in E(G^*)$.

So far, since $V(H_3)\cap V(G_{vw',2})=\{u_2\}$, we have that $e(G^*[\overline{S_1}])-e(H_2\cup H_3)\ge e(\{u_1\},V(H_3))+e(\{u_2\},V(G_{vw',2})\backslash \{u_2\})=p_3+p_2-1$. If $e(G^*[\overline{S_1}])-e(H_2\cup H_3)=p_3+p_2-1$, then $E(G^*[\overline{S_1}])= E(H_2\cup H_3) \cup E(\{u_1\},V(H_3)) \cup E(\{u_2\},V(G_{vw',2})\backslash \{u_2\})$. Let $v'$ be a vertex in $V(H_2)\cap V(G_{vw',2})$, then $K_{p_1}\cup K_{p_2}\cup K_{p_3}$ is not a subgraph of $G^*+vv'$, a contradiction. Therefore $e(G^*[\overline{S_1}])-e(H_2\cup H_3)\ge p_3+p_2$.

If $|V(G_{vw',3})\cap V(H_3)|\le p_3-2$, since $H_2\cong K_{p_2}$, $|V(G_{vw',3})\cap V(H_2)|\le p_2$. In fact, $|V(G_{vw',3})\backslash V(H_2\cup H_3)|\le 1$, otherwise there are two vertices $w_1,w_2\in V(G_{vw',3})\backslash V(H_2\cup H_3)$, then $G^*[S_1\cup \{w_1,w_2\}]\cup H_2\cup H_3\cong K_{p_1}\cup K_{p_2}\cup K_{p_3}$ in $G^*$, a contradiction. So $2\le |V(G_{vw',3})\backslash V(H_3)|\le p_2+1$. As $|V(G_{vw',3})\cap V(H_3)|+|V(G_{vw',3})\backslash V(H_3)|=|V(G_{vw',3})|=p_3$, we have that $|V(G_{vw',3})\cap V(H_3)|\times|V(G_{vw',3})\backslash V(H_3)|\ge 2(p_3-2)$, then
\begin{equation}
\begin{aligned}
e(\overline{H_1})&\ge e(H_2)+e(H_3)+|V(G_{vw',3})\cap V(H_3)|\times|V(G_{vw',3})\backslash V(H_3)|
\\&\ge\binom{p_2}{2}+\binom{p_3}{2}+2(p_3-2)
\\&\ge\binom{p_2+1}{2}+\binom{p_3+1}{2}.
\end{aligned}
\nonumber
\end{equation}
So the claim holds.
\end{proof}

Next, we proceed to prove Theorem~\ref{th1}.

If for any $u\in \overline{S\cup V(H_{vw,3})\cup \{v,w\}}$, there exists a subgraph $H=H_{vu,1}\cup H_{vu,2}\cup H_{vu,3}$ in $G[\overline{S\cup\{v,u\}}]$ with $H_{vu,i}\cong K_{p_{i+1}}$ for $i\in[2]$ and $H_{vu,3}=H_{vw,3}$, we will consider the graph $G_1=G[V(G)-V(H_{vw,3})]$, and note that $G[S]\lor (H_{vw,1}\cup H_{vw,2}\cup G[\overline{S\cup V(H_{vw})}]) \cong K_{p_1-2}\lor (K_{p_2}\cup K_{p_3}\cup I_{n+2-\sum_{i=1}^4 p_i})$ is a subgraph of $G_1$, but $K_{p_1}\cup K_{p_2}\cup K_{p_3}$ is not, and there is a subgraph $G'[S\cup \{v,u\}] \cup H_{vu,1}\cup H_{vu,2}\cong K_{p_1}\cup K_{p_2}\cup K_{p_3}$ in $G'=G_1+vu$ with $u\in V(G_1)\backslash (S\cup \{v,w\})$. By Claim~\ref{lem5}, we know that $e(G_1[V(G_1)\backslash S])\ge \binom{p_2+1}{2}+\binom{p_3+1}{2}$. It implies that $e(G[\overline{S\cup V(H_{vw,3})}])\ge \binom{p_2+1}{2}+\binom{p_3+1}{2}$. Also by Lemma~\ref{lem3}, we know that $N_{G}(u')\cap \overline{S\cup V(H_{vw,3})}\neq \emptyset$ for any $u'\in V(H_{vw,3})$, so there are at least $|V(H_{vw,3})|=p_4$ edges between $V(H_{vw,3})$ and $\overline{S\cup V(H_{vw,3})}$ in $G$. Hence
\begin{equation}
\begin{aligned}
e(G[\overline{S}])&\ge e(G_1[V(G_1)\backslash S])+e(H_{vw,3})+e(V(H_{vw,3}),\overline{S\cup V(H_{vw,3})})
\\&=\binom{p_2+1}{2}+\binom{p_3+1}{2}+\binom{p_4}{2}+p_4
\\&=\sum_{i=2}^{4}\binom{p_i+1}{2},
\end{aligned}
\nonumber
\end{equation}
a contradiction to Equation (\ref{eq1}).

Assume that there is a vertex $u\in \overline{S\cup V(H_{vw,3})\cup\{v,w\}}$ such that $|V(H_{vu,3})\cap V(H_{vw,3})|\le p_4-1$ for any choice of $H_{vu,1}\cup H_{vu,2}\cup H_{vu,3}$ in $G+vu$ with $H_{vw,i}\cong K_{p_{i+1}}$ and $i\in[3]$
\vspace{4mm}

\noindent \textbf{Case 1.} $|V(H_{vu,3})\cap V(H_{vw,3})|\le p_4-3$.

Since $|V(H_{vw,j})|=p_{j+1}<p_4-3$ for $j\in [2]$, one has that $\max \limits_{1\le i\le 3}|V(H_{vu,3})\cap V(H_{vw,i})|\le p_4-3$. Note that $|V(H_{vu,3})|=p_4\ge 9$, so we have that $\max \limits_{1\le i\le 3}|V(H_{vu,3})\cap V(H_{vw,i})|\ge 3$. Let $a=\max \limits_{1\le i\le 3}|V(H_{vu,3})\cap V(H_{vw,i})|$. Then $3\le a\le p_4-3$ and $p_4-a\ge 3$. It implies that $a(p_4-a)\ge 3(p_4-3)$.Hence
\begin{equation}
\begin{aligned}
e(G[\overline{S}])&\ge e(H_{vw})+a(p_4-a)
\\&\ge\sum_{i=2}^{4}\binom{p_i}{2}+3(p_4-3)>\sum_{i=2}^{4}\binom{p_i+1}{2},
\end{aligned}
\nonumber
\end{equation}
a contradiction to Equation (\ref{eq1}).
\vspace{4mm}

\noindent \textbf{Case 2.} $|V(H_{vu,3})\cap V(H_{vw,3})| = p_4-2$.

There are two vertices in $V(H_{vw,3})\backslash V(H_{vu,3})$, namely $v_1$ and $v_2$. If $v_1,v_2\notin V(H_{vu})$, then $G[S,\{v_1,v_2\}]\cup H_{vu,1}\cup H_{vu,2}\cup H_{vu,3}\cong K_{p_1}\cup K_{p_2} \cup K_{p_3}\cup K_{p_4}$ in $G$, a contradiction. It implies that there is at least one vertex of $\{v_1,v_2\}$ in $V(H_{vu,i})$ for $i\in [2]$. If $v_1,v_2 \in V(H_{vu,1})\cup V(H_{vu,2})$, then
\begin{equation}
\begin{aligned}
e(G[\overline{S}])&\ge e(H_{vw})+e(\{v_1,v_2\},(V(H_{vu,1})\cup V(H_{vu,2}))\backslash\{v_1,v_2\})
\\&+|V(H_{vu,3})\backslash V(H_{vw,3})|\times|V(H_{vu,3})\cap V(H_{vw,3})|
\\&\ge\sum_{i=2}^{4}\binom{p_i}{2}+2(p_2-2)+2(p_4-2)>\sum_{i=2}^{4}\binom{p_i+1}{2},
\end{aligned}
\nonumber
\end{equation}
a contradiction to Equation (\ref{eq1}). So there is only one vertex of $\{v_1,v_2\}$ in $V(H_{vu,1})\cup V(H_{vu,2})$, say $v_1\in V(H_{vu,1})\cup V(H_{vu,2})$. By Lemma~\ref{lem3} and $v_2\in V(H_{vw,3})$, we know that $v_2$ has a neighbor in $\overline{S}\backslash V(H_{vw,3})$. Hence
\begin{equation}
\begin{aligned}
e(G[\overline{S}])&\ge e(H_{vw})+|V(H_{vu,3})\backslash V(H_{vw,3})|\times|V(H_{vu,3})\cap V(H_{vw,3})|
\\&+e(\{v_1\},(V(H_{vu,1})\cup V(H_{vu,2}))\backslash\{v_1\})+e(\{v_2\},\overline{S}\backslash V(H_{vw,3}))
\\&\ge\sum_{i=2}^{4}\binom{p_i}{2}+2(p_4-2)+(p_2-1)+1\ge\sum_{i=2}^{4}\binom{p_i+1}{2},
\end{aligned}
\nonumber
\end{equation}
a contradiction to Equation (\ref{eq1}).
\vspace{4mm}

\noindent \textbf{Case 3.} $|V(H_{vu,3})\cap V(H_{vw,3})| = p_4-1$.

Since $H_{vu,3}\cong K_{p_4}$ and $|V(H_{vu,3})\cap V(H_{vw,3})| = p_4-1$, one has that there is a vertex in $V(H_{vu,3})\backslash V(H_{vw,3})$, namely $v_1$. Let $V(H_{vw,3})\backslash V(H_{vu,3})=\{v_2\}$, then $V(H_{vw,3})\cap V(H_{vu,i})=\{v_2\}$ for some $i\in [2]$, otherwise $H'_{vu,1}\cup H'_{vu,2}\cup H'_{vu,3}=H_{vu,1}\cup H_{vu,2}\cup G[V(H_{vu,3})\backslash \{v_1\})\cup \{v_2\}]$ is also a subgraph of $G[\overline{S\cup \{v,u\}}]$ with $H'_{vu,3}=H_{vw,3}$, a contradiction to the choice of $u$.
\vspace{4mm}

\noindent \textbf{Subcase 3.1.} $i=1$.

For simplicity, let $V(H_{vu,1})\cap V(H_{vw,j})=L_{1j}$ and $|V(H_{vu,1})\cap V(H_{vw,j})|={\ell}_{1j}$ for $j\in[3]$ (see Figure~\ref{fig2}). Note that ${\ell}_{13}=1$ and ${\ell}_{11}+{\ell}_{12}+{\ell}_{13}\le p_2$, so ${\ell}_{11}+{\ell}_{12}\le p_2-1$. If ${\ell}_{11}+{\ell}_{12}< p_2-2$, then there are two adjacent vertices $y$ and $z$ in $G[V(H_{vu,1}) \backslash V(H_{vw})]$. Hence, $G[S \cup \{y, z\}] \cup H_{vw}$ is a subgraph isomorphic to $K_{p_1}\cup \dots \cup K_{p_4}$ in $G$, a contradiction. So ${\ell}_{11}+{\ell}_{12}\ge p_2-2$.
 \begin{figure}[H]
	\centering
	\includegraphics[width=0.4\linewidth]{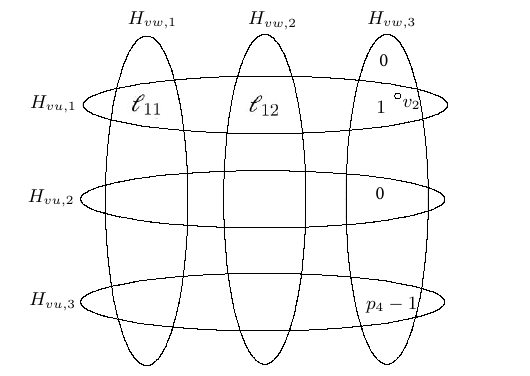} 
	\caption{The illustration of Subcase 3.1 with $H_{vu,i}\cong K_{p_{i+1}}\cong H_{vw,i}$ for $i\in[3]$.}
	\label{fig2}
\end{figure}
\vspace{4mm}

\noindent \textbf{Subcase 3.1.1.} ${\ell}_{12}\ge 1$.

Recall that $u\in \overline{S\cup V(H_{vw,3})\cup\{v,w\}}$, we will prove that $u\in V(H_{vw,1})\cup V(H_{vw,2})$. Otherwise $u\in V(\overline{S\cup H_{vw}\cup \{v,w\}})$. Thus $G'[S\cup \{u,v\}]\cup H_{vw}$ is a subgraph of $G'=G+vu$ which is isomorphic to $K_{p_1}\cup K_{p_2} \cup K_{p_3}\cup K_{p_4}$ satisfying that $H_{vu,3}=H_{vw,3}$, a contradiction to the choice of $u$.

First, we consider $u\in V(H_{vw,1})$. It can be proved that ${\ell}_{11}\ge 1$, otherwise we have ${\ell}_{12}\ge p_2-2$. And since ${\ell}_{12}\le p_2-1$ and $H_{vw,2}\cong K_{p_3}$, one has that $|V(H_{vw,2})\backslash L_{12}|\ge p_3-p_2+1$. Therefore
\begin{equation}
\begin{aligned}
e(G[\overline{S}])&\ge e(H_{vu})+e(\{v_2\},V(H_{vw,3})\backslash\{v_2\})+e(\{u\},V(H_{vw,1})\backslash\{u\})+{\ell}_{12}|V(H_{vw,2})\backslash L_{12}|
\\&\ge \sum_{i=2}^{4}\binom{p_i}{2}+(p_4-1)+(p_2-1)+(p_2-2)(p_3-p_2+1)\ge \sum_{i=2}^{4}\binom{p_i+1}{2},
\end{aligned}
\nonumber
\end{equation}
a contradiction to Equation (\ref{eq1}). If $V(H_{vw,2})\backslash V(H_{vu})\neq \emptyset$, there is a vertex $x_1\in V(H_{vw,2})\backslash V(H_{vu})$, since ${\ell}_{11}\ge 1$ and ${\ell}_{12}\ge 1$, we have that
\begin{equation}
\begin{aligned}
e(G[\overline{S}])-e(H_{vu})&\ge e(\{v_2\},V(H_{vw,3})\backslash\{v_2\})+e(\{u\},V(H_{vw,1})\backslash\{u\})+{\ell}_{11}|V(H_{vw,1})\backslash (\{u\}\cup L_{11})|
\\&+e(\{x_1\},V(H_{vw,2})\backslash\{x_1\})+{\ell}_{12}|V(H_{vw,2})\backslash (\{x_1\}\cup L_{12})|
\\&\ge (p_4-1)+(p_2-1)+(p_2-2)+(p_3-1)+(p_3-2)\ge p_2+p_3+p_4,
\end{aligned}
\nonumber
\end{equation}
a contradiction to Equation (\ref{eq1}). So $V(H_{vw,2})\backslash V(H_{vu})=\emptyset$. By Lemma~\ref{lem6} and $|V(H_{vu})\cup \{u\}|=p_2+p_3+p_4+1$, there is a vertex $x_2\in \overline{S\cup V(H_{vu})\cup \{v,u\}}$ with $d_{G[\overline{S}]}(x_2)\ge p_2-1$. It is true that $V(H_{vw,1})\backslash (\{u\}\cup V(H_{vu}))=\emptyset$, otherwise there is a vertex $u'\in V(H_{vw,1})\backslash (\{u\}\cup V(H_{vu}))$ and $G[S\cup \{u,u'\}]\cup H_{vu}\cong K_{p_1}\cup K_{p_2} \cup K_{p_3}\cup K_{p_4}$ in $G$, a contradiction. Also recall that ${\ell}_{13}=1$ and $|V(H_{vu,3})\cap V(H_{vw,3})|=p_4-1$, since $H_{vw,3}\cong K_{p_4}$, one has that $V(H_{vw,3})\backslash V(H_{vu})=\emptyset$. From the above, we know that $x_2\notin S\cup V(H_{vw})\cup V(H_{vu})\cup \{v,u\}$. Recall that ${\ell}_{11}\ge 1$ and ${\ell}_{12}\ge 1$. Thus
\begin{equation}
\begin{aligned}
e(G[\overline{S}])&\ge e(H_{vu})+e(\{v_2\},V(H_{vw,3})\backslash\{v_2\})+e(\{u\},V(H_{vw,1})\backslash\{u\})+{\ell}_{11}|V(H_{vw,1})\backslash (\{u\}\cup L_{11})|
\\&+{\ell}_{12}|V(H_{vw,2})\backslash L_{12}|+d_{G[\overline{S}]}(x_2)
\\&\ge\sum_{i=2}^{4}\binom{p_i}{2}+(p_4-1)+(p_2-1)+(p_2-2)+(p_3-1)+(p_2-1)>\sum_{i=2}^{4}\binom{p_i+1}{2},
\end{aligned}
\nonumber
\end{equation}
a contradiction to Equation (\ref{eq1}).

Next we consider $u\in V(H_{vw,2})$. It can be proved that ${\ell}_{11}\ge 1$, otherwise we have ${\ell}_{12}\ge p_2-2$. And since ${\ell}_{12}\le p_2-1$ and $H_{vw,2}\cong K_{p_3}$, one has that $|V(H_{vw,2})\backslash (\{u\}\cup L_{12})|\ge p_3-p_2$. Therefore
\begin{equation}
\begin{aligned}
e(G[\overline{S}])&\ge e(H_{vu})+e(\{v_2\},V(H_{vw,3})\backslash\{v_2\})+e(\{u\},V(H_{vw,2})\backslash\{u\})+{\ell}_{12}|V(H_{vw,2})\backslash (\{u\}\cup L_{12})|
\\&\ge\sum_{i=2}^{4}\binom{p_i}{2}+(p_4-1)+(p_3-1)+(p_2-2)(p_3-p_2)>\sum_{i=2}^{4}\binom{p_i+1}{2},
\end{aligned}
\nonumber
\end{equation}
a contradiction to Equation (\ref{eq1}). Note that ${\ell}_{11}\ge 1$ and ${\ell}_{12}\ge 1$, then
\begin{equation}
\begin{aligned}
e(G[\overline{S}])&\ge e(H_{vu})+e(\{v_2\},V(H_{vw,3})\backslash\{v_2\})+e(\{u\},V(H_{vw,2})\backslash\{u\})+{\ell}_{12}|V(H_{vw,2})\backslash (\{u\}\cup L_{12})|
\\&+{\ell}_{11}|V(H_{vw,1})\backslash L_{11}|
\\&\ge\sum_{i=2}^{4}\binom{p_i}{2}+(p_4-1)+(p_3-1)+(p_3-2)+(p_2-1)>\sum_{i=2}^{4}\binom{p_i+1}{2},
\end{aligned}
\nonumber
\end{equation}
a contradiction to Equation (\ref{eq1}).
\vspace{4mm}

\noindent \textbf{Subcase 3.1.2.} ${\ell}_{12}=0$.

Recall that $p_2-2\le {\ell}_{11}+{\ell}_{12}\le p_2-1$. Then $p_2-2\le {\ell}_{11}\le p_2-1$.

First we consider ${\ell}_{11}=p_2-1$. Recall that $v_1\in V(H_{vu,3})\backslash V(H_{vw,3})$.

If $v_1\in V(H_{vw,1})$, then $G[(V(H_{vu,1})\backslash \{v_2\})\cup \{v_1\}]\cup H_{vu,2}\cup G[V(H_{vu,3})\backslash \{v_1\})\cup \{v_2\}]=H_{vw,1}\cup H_{vu,2}\cup H_{vw,3}$ is also a copy of $K_{p_2}\cup K_{p_3}\cup K_{p_4}$ in $G[\overline{S\cup \{v,u\}}]$, a contradiction to the arbitrariness of $u$'s choice (see Figure~\ref{fig3}).
\begin{figure}[H]
	\centering
	\includegraphics[width=0.4\linewidth]{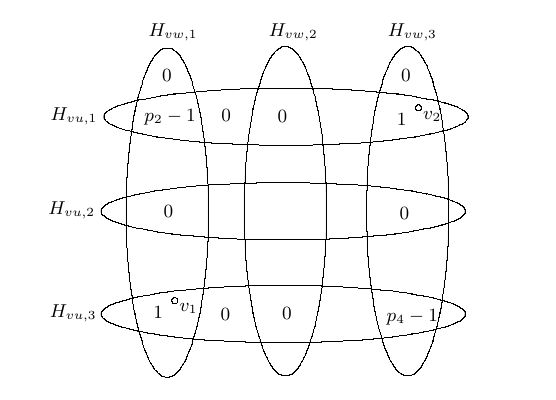} 
	\caption{The illustration of Subcase 3.1.2 with ${\ell}_{11}=p_2-1$, $v_1\in V(H_{vw,1})$ and $H_{vu,i}\cong K_{p_{i+1}}\cong H_{vw,i}$ for $i\in[3]$.}
	\label{fig3}
\end{figure}

If $v_1\in V(H_{vw,2})$, there is at most one vertex in $V(H_{vw,2}) \backslash V(H_{vu})$, otherwise assume that $u_1,u_2\in V(H_{vw,2}) \backslash V(H_{vu})$, then $G[S\cup \{u_1,u_2\}]\cup H_{vu,1}\cup H_{vu,2}\cup H_{vu,3}\cong K_{p_1}\cup K_{p_2} \cup K_{p_3}\cup K_{p_4}$ in $G$, a contradiction. It implies that $p_3-2 \le |V(H_{vw,2})\cap V(H_{vu,2})|\le p_3-1$ and therefore $|V(H_{vw,2})\cap V(H_{vu,2})|\times|V(H_{vu,2})\backslash V(H_{vw,2})|\ge \min \{p_3-1,2(p_3-2)\}=p_3-1$. Since there is at most one vertex $y$ of $\overline{S\cup V(H_{vw})}$ in $V(H_{vu,2})\backslash V(H_{vw})$ and $(V(H_{vu,1})\cup V(H_{vu,3}))\backslash V(H_{vw})=\emptyset$, by Lemma~\ref{lem6}, we know that there is a vertex $z$ in $\overline{S\cup V(H_{vw})\cup \{v,w,y\}}$ which has at least $p_2-1$ neighbors in $G[\overline{S}]$. Hence
\begin{equation}
\begin{aligned}
e(G[\overline{S}])&\ge e(H_{vw})+e(\{v_2\},V(H_{vu,1})\cap V(H_{vw,1}))+|V(H_{vw,2})\cap V(H_{vu,2})|\times|V(H_{vu,2})\backslash V(H_{vw,2})|
\\&+e(\{v_1\},V(H_{vu,3})\cap V(H_{vw,3}))+e(\{z\},\overline{S})
\\&\ge\sum_{i=2}^{4}\binom{p_i}{2}+(p_2-1)+(p_3-1)+(p_4-1)+(p_2-1)>\sum_{i=2}^{4}\binom{p_i+1}{2},
\end{aligned}
\nonumber
\end{equation}
a contradiction to Equation (\ref{eq1}).

If $v_1\in V(H_{vu,3})\backslash V(H_{vw})$, by Lemma~\ref{lem3}, we know that $N_{G}(u')\cap (\overline{S\cup V(H_{vw,2})})\neq \emptyset$ for any $u'\in V(H_{vw,2})$, so there is a set $L$ of edges between $V(H_{vw,2})$ and $\overline{S\cup V(H_{vw,2})}$ in $G$ with $|L|\ge |V(H_{vw,2})|=p_3$. Therefore we have that $e(G[\overline{S}])-e(H_{vw})\ge e(\{v_2\},L_{11})+|L|+e(\{v_1\},V(H_{vw,3})\cap V(H_{vu,3}))\ge p_2-1+p_3+p_4-1=p_2+p_3+p_4-2$. If $e(G[\overline{S}])-e(H_{vw})\le p_2+p_3+p_4-1$, we consider the graph $G+vx_1$ with $x_1\in V(H_{vw,3})\cap V(H_{vu,3})$ and there is a subgraph $H'$ which is isomorphic to $K_{p_4}$ in $G[\overline{S\cup \{v,x_1\}}]$. Since only the vertices in $V(H_{vw,3})\cup \{v_1\}$ can be in $V(H')$, $G[\{v_1\}\cup (V(H_{vw,3})\backslash \{x_1\})]$ is the only possible $H'$. It implies that $v_1v_2\in E(G[\overline{S}])$. Note that $v_1v_2\notin E(\{v_2\},L_{11})\cup L \cup E(\{v_1\},V(H_{vw,3})\cap V(H_{vu,3}))$, so $e(G[\overline{S}])-e(H_{vw})\ge p_2+p_3+p_4-1$. Therefore $e(G[\overline{S}])-e(H_{vw})= p_2+p_3+p_4-1$ and $E(G[\overline{S}])=E(H_{vw})\cup E(\{v_2\},L_{11})\cup L\cup E(\{v_1\},V(H_{vw,3})\cap V(H_{vu,3}))\cup \{v_1v_2\}$, so there is no subgraph in $G+vx_2$ which is isomorphic to $K_{p_1}\cup K_{p_2}\cup K_{p_3}\cup K_{p_4}$ for $x_2\in V(H_{vw,1})\cap V(H_{vu,1})$, a contradiction. Hence $e(G[\overline{S}])-e(H_{vw})\ge p_2+p_3+p_4$, a contradiction to Equation (\ref{eq1}).

Next we consider ${\ell}_{11}=p_2-2$. Note that ${\ell}_{12}=0$ and ${\ell}_{13}=1$, then there is a vertex $y'$ in $V(H_{vu,1})\backslash V(H_{vw})$. By Lemma~\ref{lem3}, we know that $N_{G}(u')\cap (\overline{S\cup V(H_{vw,2})})\neq \emptyset$ for any $u'\in V(H_{vw,2})$, so there are at least $|V(H_{vw,2})|=p_3$ edges between $V(H_{vw,2})$ and $\overline{S\cup V(H_{vw,2})}$ in $G$. It is possible that $v_1\in V(H_{vw,2})$, so there is a set $L_1$ of edges between $V(H_{vw,2})\backslash \{v_1\}$ and $\overline{S\cup V(H_{vw,2})}$ in $G$ with $|L_1|\ge p_3-1$. Therefore we have that $e(G[\overline{S}])-e(H_{vw})\ge e(\{v_2\},V(H_{vu,1})\backslash\{v_2\})+e(\{y'\},V(H_{vw,1})\cap V(H_{vu,1}))+|L_1|+e(\{v_1\},V(H_{vw,3})\backslash\{v_2\})=p_2-1+p_2-2+p_3-1+p_4-1\ge p_2+p_3+p_4-1$. If $e(G[\overline{S}])-e(H_{vw})=p_2+p_3+p_4-1$, then $E(G[\overline{S}])=E(H_{vw})\cup E(\{v_2\},V(H_{vu,1})\backslash\{v_2\})\cup E(\{y'\},V(H_{vw,1})\cap V(H_{vu,1})) \cup L_1 \cup E(\{v_1\},V(H_{vw,3})\backslash\{v_2\})$ with $|L_1| = p_3-1$. We will consider the graph $G+vy_1$ with $y_1\in V(H_{vw,3})\cap V(H_{vu,3})$ and there is a subgraph $H'$ which is isomorphic to $K_{p_4}$ in $G[\overline{S\cup \{v,y_1\}}]$. Since only the vertices in $V(H_{vw,3})\cup \{v_1\}$ can be in $V(H')$, $G[\{v_1\}\cup (V(H_{vw,3})\backslash \{y_1\})]$ is the only possible $H'$. It implies that $v_1v_2\in E(G[\overline{S}])$. So $e(G[\overline{S}])-e(H_{vw})\ge p_2+p_3+p_4$, a contradiction to Equation (\ref{eq1}).
\vspace{4mm}

\noindent \textbf{Subcase 3.2.} $i=2$.

First we consider that $V(H_{vu,2})\backslash\{v_2\} \subseteq V(H_{vw,2})$, and it implies that $|V(H_{vw,2})\cap V(H_{vu,2})|=p_3-1$. Recall that $v_1\in V(H_{vu,3})\backslash V(H_{vw,3})$. If $v_1\in V(H_{vw,1})\cup V(H_{vw,2})$, by Lemma~\ref{lem6}, we know that there are two vertices $y$ and $z$ in $\overline{S\cup V(H_{vw})}$ with $d_{G[\overline{S}]}(y)\ge p_2-1$ and $d_{G[\overline{S}]}(z)\ge p_2-1$. Hence
\begin{equation}
\begin{aligned}
e(G[\overline{S}])&\ge e(H_{vw})+e(\{v_2\},V(H_{vu,2})\cap V(H_{vw,2}))+e(\{v_1\},V(H_{vu,3})\cap V(H_{vw,3}))
\\&+e(\{y\},\overline{S})+e(\{z\},\overline{S})
\\&\ge\sum_{i=2}^{4}\binom{p_i}{2}+(p_3-1)+(p_4-1)+2(p_2-1)>\sum_{i=2}^{4}\binom{p_i+1}{2},
\end{aligned}
\nonumber
\end{equation}
a contradiction to Equation (\ref{eq1}).

If $v_1\in V(H_{vu,3})\backslash V(H_{vw})$ (see Figure~\ref{fig4}), by Lemma~\ref{lem3}, we know that $N_{G}(u')\cap (\overline{S\cup V(H_{vw,1})})\neq \emptyset$ for any $u'\in V(H_{vw,1})$, so there is a set $L$ of edges between $V(H_{vw,1})$ and $\overline{S\cup V(H_{vw,1})}$ in $G$ with $|L|\ge |V(H_{vw,1})|=p_2$. Therefore we have that $e(G[\overline{S}])-e(H_{vw})\ge |L|+e(\{v_2\},V(H_{vu,2})\cap V(H_{vw,2}))+e(\{v_1\},V(H_{vw,3})\cap V(H_{vu,3}))=p_2+p_3-1+p_4-1=p_2+p_3+p_4-2$. If $e(G[\overline{S}])-e(H_{vw})\le p_2+p_3+p_4-1$, we will consider the graph $G+vx_1$ with $x_1\in V(H_{vw,3})\cap V(H_{vu,3})$ and there is a subgraph $H'$ which is isomorphic to $K_{p_4}$ in $G[\overline{S\cup \{v,x_1\}}]$. Since only the vertices in $V(H_{vw,3})\cup \{v_1\}$ can be in $V(H')$, $G[\{v_1\}\cup (V(H_{vw,3})\backslash \{x_1\})]$ is the only possible $H'$. It implies that $v_1v_2\in E(G[\overline{S}])$. Note that $v_1v_2\notin L \cup  E(\{v_2\},V(H_{vu,2})\cap V(H_{vw,2}))\cup E(\{v_1\},V(H_{vw,3})\cap V(H_{vu,3}))$, so $e(G[\overline{S}])-e(H_{vw})\ge p_2+p_3+p_4-1$. Therefore $e(G[\overline{S}])-e(H_{vw})= p_2+p_3+p_4-1$ and $E(G[\overline{S}])=E(H_{vw})\cup L \cup  E(\{v_2\},V(H_{vu,2})\cap V(H_{vw,2}))\cup E(\{v_1\},V(H_{vu,3})\cap V(H_{vw,3}))\cup \{v_1v_2\}$ with $|L|=p_2$, so there is no subgraph in $G+vx_2$ which is isomorphic to $K_{p_1}\cup K_{p_2}\cup K_{p_3}\cup K_{p_4}$ for $x_2\in V(H_{vw,2})\cap V(H_{vu,2})$, a contradiction. Hence $e(G[\overline{S}])-e(H_{vw})\ge p_2+p_3+p_4$, a contradiction to Equation (\ref{eq1}).
\begin{figure}[H]
	\centering
	\includegraphics[width=0.4\linewidth]{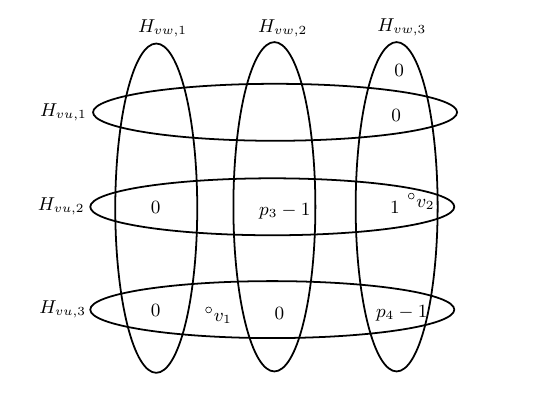} 
	\caption{The illustration of Subcase 3.2 with $V(H_{vu,2})\backslash\{v_2\} \subseteq V(H_{vw,2})$, $v_1\in V(H_{vu,3})\backslash V(H_{vw})$ and $H_{vu,i}\cong K_{p_{i+1}}\cong H_{vw,i}$ for $i\in[3]$.}
	\label{fig4}
\end{figure}

Next we consider that $V(H_{vu,2})\backslash\{v_2\} \nsubseteq V(H_{vw,2})$. It implies that $|V(H_{vu,2})\backslash V(H_{vw})|\ge 1$. Since $|V(H_{vu,2})\backslash\{v_2\}|=p_3-1$, one has that $|E(G[V(H_{vu,2})\backslash\{v_2\}])\backslash E(H_{vw})|\ge p_3-2$. Therefore,
\begin{equation}
\begin{aligned}
e(G[\overline{S}])&\ge e(H_{vw})+e(\{v_2\},V(H_{vu,2})\cap V(H_{vw,2}))+e(\{v_1\},V(H_{vu,3})\cap V(H_{vw,3}))
\\&+|E(G[V(H_{vu,2})\backslash\{v_2\}])\backslash E(H_{vw})|
\\&\ge\sum_{i=2}^{4}\binom{p_i}{2}+(p_3-1)+(p_4-1)+(p_3-2)\ge\sum_{i=2}^{4}\binom{p_i+1}{2},
\end{aligned}
\nonumber
\end{equation}
a contradiction to Equation (\ref{eq1}).
\qed
\section{Remark}\label{sec 5}
 In Section~\ref{sec 3}, we have proved that $H(n; p_1, p_2, \dots, p_t)$ is $K_{p_1} \cup \dots \cup K_{p_t}$-saturated if and only if $p_{i+1}-p_i \ge p_1$ or $p_{i+1}=p_i$ for $2\le i\le t-1$. Let the set of all $n$-vertex $H$-saturated graph with $sat(n, H)$ edges is denoted by $Sat(n, H)$. We finish our discussion with a question related to $H(n; p_1, p_2, \dots, p_t)$.

\begin{problem}
When $H(n; p_1, p_2, \dots, p_t)$ is $K_{p_1} \cup \dots \cup K_{p_t}$-saturated, does the result that \\$H(n; p_1, p_2, \dots, p_t) \in Sat(n,K_{p_1} \cup K_{p_2} \cup \dots \cup K_{p_t})$ holds?
\end{problem}
\section*{Acknowledgments}
This paper is supported by the National Natural Science Foundation of China (No.~12331013 and No.~12161141005); and by Beijing Natural Science Foundation (No.~1244047), China Postdoctoral Science Foundation (No.~2023M740207).

\end{document}